\documentclass[12pt, reqno]{amsart}
\usepackage{amsfonts}
\usepackage{amsmath,amsthm,mathrsfs}
\usepackage{amssymb}
\usepackage{enumitem}
\usepackage{tikz}
\usepackage{bm}
\usepackage{url}
\usepackage[normalem]{ulem}

\usepackage[colorlinks=true,linkcolor=blue,citecolor=blue,urlcolor=blue,pdfborder={0 0 0}]{hyperref}

\usepackage{soul}

\usepackage{graphicx}
\usepackage{subcaption}

\usepackage{xcolor}
\usepackage[capitalize]{cleveref}

\usepackage{biblatex}
\theoremstyle{plain}

\newtheorem{thm}{Theorem}[section]

\newtheorem{prop}[thm]{Proposition}

\newtheorem{lemma}[thm]{Lemma}

\newtheorem{cor}[thm]{Corollary}
\newtheorem{claim}[thm]{Claim}

\theoremstyle{definition}
\newtheorem{remark}[thm]{Remark}

\usepackage[legalpaper, margin=1.2in]{geometry}

\usepackage{fancyvrb}
\usepackage{mathtools}

\newcommand{\eps}{\epsilon}

\newcommand{\R}{\mathbb{R}}

\newcommand{\Reff}{R_\mathrm{eff}}

\renewcommand{\P}{\mathbb{P}}

\DeclareMathOperator{\E}{\mathbb{E}}

\newcommand{\RR}{\mathbb{R}}

\newcommand{\Z}{\mathbb{Z}}

\newcommand{\lr}{\leftrightarrow}

\newcommand{\pot}{\Psi}

\newcommand{\eq}{equation}
\newcommand{\be}{\begin{\eq}}
\newcommand{\ee}{\end{\eq}}

\renewcommand{\and}{\mathrm{\ and\ }}

\renewcommand{\o}{o}
\newcommand{\Y}{Y}
\newcommand{\q}{q}

\newcommand{\F}{\mathcal{F}}

\newcommand{\G}{{\sf G}}
\newcommand{\g}{{\sf g}}

\begin{document}

\title{Potentials in recurrent networks: a survey}
\author{Asaf Nachmias\footnote{Department of Mathematical Sciences, Tel Aviv University, Israel. }}
\author{Yuval Peres\footnote{Beijing Institute of Mathematical Sciences and Applications, Beijing, China. }}

\begin{abstract}   A nonnegative function on the vertices of an infinite graph $G$ which vanishes at a distinguished vertex $\o$, has Laplacian $1$ at $\o$, and is harmonic at all other vertices is called a potential. We survey basic properties of potentials in recurrent networks. In particular, we show that potentials are Lipschitz with respect to the effective resistance metric, and   
 if the potential is unique, then there is a determinantal formula for the harmonic measures from infinity.  We also infer from the von Neumann minimax theorem that there always exists a potential tending to infinity.  

\end{abstract}

\begingroup
\def\uppercasenonmath#1{} 
\let\MakeUppercase\relax 
\maketitle
\endgroup


\section{Introduction}\label{sec:intro}

In a transient network, the importance of the Green function is largely due to its Laplacian being a Dirac measure. In recurrent networks, where there is no Green function, this role is played by potentials. Let $G=(V,E)$ be an infinite, connected, locally finite graph endowed with positive edge weights  $\big\{c_{xy}\}_{xy \in E} $. Define the network {\bf Laplacian} of a function $f:V\to\RR$ by 
$$ \Delta f (v) = \sum_{x \sim v} c_{vx} \big(f(x) - f(v)\big) \, .$$

Fix a vertex $\o \in V$. A {\bf potential} on the {\em rooted network} $(G,c,\o )$ is a function $h:V\to [0,\infty)$ that is harmonic off $\o$ (i.e.,  $\Delta h(v) = 0$ for every vertex $v\neq \o$), 
normalized so that $h(\o)=0$ and $\Delta h(\o) =1$.
In every such rooted network there is   at least one potential, see \cref{claim:compact}. In $\mathbb{Z}^2$ the unique potential is a classical object, see, e.g., \cite[Chapter 12]{Spitzer} and \cite[Chapter 4.4]{LawlerLimic}.
Recently there has been renewed interest (see Berestycki and van Engelenburg~\cite{BereEngel} and van Engelenburg and Hutchcroft~\cite{EngelHutch}) in recurrent potential theory, motivated by close connections to other natural notions: 
\begin{itemize}
\item Random walks conditioned to tend to infinity;
\item Harmonic measures from infinity;
\item Uniform spanning trees (USTs).
\end{itemize}
In this survey, we study these connections and other basic properties of potentials in recurrent networks. Given a network $(G,c)$, the corresponding {\em network random walk} defined by $p(x,y)=c_{xy}/c_x$ is reversible, with reversing measure $c_x=\sum_{z\sim x}c_{xz}$. Any potential $h$ defines a Markov chain on the state space $S_h = \{v : h(v)>0\}$ via the Doob $h$-transform that has transition probabilities
$$ p^h(x,y) = {p(x,y) h(y) \over h(x)} \, .$$
The corresponding Markov chain $\{Y_n\}_{n \geq 0}$, called the $h$-process, is reversible, with reversing measure $c_x h(x)^2$ and edge conductances $\big\{c_{xy} h(x)h(y)\big\}_{xy \in E}$. 
The $h$-process is transient, see \cref{lem:hProcessInfty}. The main results we present pertaining  to a recurrent rooted network $(G,c,\o)$ are  (references are given in the relevant sections): 
\begin{enumerate}
\item[{\bf (a)}] Every potential is Lipschitz with respect to the effective resistance metric (\cref{cor:PotentialLipschitz}).
\item[{\bf (b)}] There is a unique potential if and only if harmonic measures from infinity exist; in this case there is a determinantal formula for the harmonic measures (\cref{thm:PotentialHarmonicMeasure}). 
\item[{\bf (c)}] If $\inf_e c_e >0$, then for every potential $h$ there exists an exhaustion $\{D_n\}$ of $G$, such that the network random walks conditioned to exit $D_n$ before reaching $\o$,  converge weakly to the $h$-process (\cref{cor:exhaustion}). 
\item[{\bf (d)}] There always exists a potential tending to infinity (\cref{conj:main}).
\item[{\bf (e)}] Let $\g_\o(x,y)$ denote the Green density  for the random walk killed at $\o$, see \eqref{eq:greend} and fix a potential $h$. Then $\g_\o(\cdot,Y_n)$ converges along the $h$-process $\{Y_n\}$ to a random potential with expectation $h$ (\cref{thm:convergence}). 
\item[{\bf (f)}]  Suppose that $h_1$ and $ h_2$ are different extremal potentials. Then the trajectories of the $h_1$-process  and the $h_2$-process  intersect finitely often a.s. (\cref{cor:FiniteIntersection}).  
\end{enumerate}

We also clarify the connection between the UST in a recurrent network and the corresponding potentials.
It was shown in \cite[Theorem 14.2]{BLPS} that if the UST in $(G,c)$ has one end a.s., then the network admits harmonic measures from infinity, hence by (b), for each root $\o$ there exists a unique potential $h$. 
In  \cref{lem:Not1Ended}, we prove that  under the same assumption on the UST, two independent trajectories of the $h$-process have infinite intersection almost surely.

In \cref{sec:PotentialsUST} we show that both of these implications cannot be reversed. In particular, for all $d \ge 5$, the lattice $\Z^d$ rooted at $\o$ and equipped with suitable conductances has a unique potential $h$, yet with probability $1$, two independent trajectories of the $h$-process have finite intersection and the UST has multiple ends; see \cref{prop:counterexample1}. This construction can be adapted to answer negatively a question raised in~\cite[Section 7]{BereEngel}, whether a unique potential implies infinite intersection for two independent $h$-process trajectories; see \cref{prop:countergraph}.




\section{Basic properties of dipoles and potentials}\label{sec:basic}

The {\bf Green kernel} $\G_\o(x,y)$ of the random walk $\{X_n\}_{n \geq 0}$ killed at $\o$ is defined by
$$ \G_\o(x,y) := \sum_{n=0}^\infty \P_x( X_n = y, \,\, \tau_\o > n) \, ,$$
where $\tau_\o = \min\{ n \geq 0 : X_n = \o\}$. The corresponding {\bf Green density},  
\be \label{eq:greend}
{\sf g}_\o (x,y) = \frac{\G_\o(x,y)}{c_y} \, ,
\ee
is symmetric by reversibility. A {\bf dipole} from $\o$ to another vertex $y$ in a network $(G,c)$ is a function $f:V\to [0,\infty)$, for which $f(\o)=0$ and 
$$\forall v \in V \qquad \Delta f (v) = {\bf 1}_\o(v)  - {\bf 1}_y(v) \, .$$

\begin{claim}\label{claim:GreenIsDipole} In a recurrent network, the Green density $g_\o(\cdot,y)$ for the killed random walk is the unique dipole from $\o$ to $y$. Moreover,  $\sup_y g_\o(x,y)<\infty$ for each $x \in V$.   
\end{claim}
\begin{proof}
First observe that $g_\o(\cdot,y)$ is harmonic off $\{\o,y\}$ and satisfies $\g_\o(\o, y)=0$. Second, $\Delta \g_\o(\cdot,y) (y) =-1$, since
\begin{equation} \label{eq:laplace-y}
     \sum_{x \sim y} c_{xy} (\g_\o(y,y)-\g_\o(x,y)) = \G_\o(y,y) - \sum_{x \sim y} {c_{xy} \over c_y} \G_\o(x,y) = 1 \, .
    \end{equation}
Also,
\be \label{eq:dipole11}  \Delta \g_\o(\cdot,y)(\o) = \sum_{x\sim \o}   c_{\o x} g_\o(x,y)= \sum_{x\sim \o} \G_o(y,x) \frac{c_{\o x}}{c_x} =1 \,,
\ee
where the last equality is obtained by summing, over  $n \ge 0$, the identity
$$
\sum_{x\sim \o}  \P_y( X_{n} = x  , \,\, \tau_\o > n)\frac{c_{\o x}}{c_x} =  \P_y( X_{n+1} = \o, \,\, \tau_\o > n)  = \P_y(   \tau_\o =n+1)  \, 
$$
and using recurrence. Thus, $g_\o(\cdot,y)$ is a dipole. 

Next, we verify uniqueness. For any dipole $f$ from $\o$ to $y$,  the difference
$f-g_\o(\cdot,y)$ is a harmonic function on $V$ that vanishes at $\o$ and is bounded below by $-\g_\o(y,y)$; therefore, it vanishes on $V$. 

Lastly, the identity \eqref{eq:dipole11}  implies that $\g_\o(x,y) \le c_{ox}^{-1}$ for every $x \sim \o$. More generally, since $\g_\o(\cdot,y)$ is superharmonic off $\o$, for any $x \ne \o$ we have $\g_\o(x,y) \le \prod_{i=1}^\ell c_{x_{i-1}x_i}^{-1}$, where $x_0,\dots,x_\ell$ is a simple path from $\o$ to $x$. 
\end{proof}

\begin{claim}\label{claim:compact} The space of potentials on a recurrent rooted $(G,c,\o)$ is a nonempty, compact, convex subset of $\R^V$. Moreover, if there is a unique potential $h$, then $\g_\o(\cdot, y_n )$ converges to $h$ pointwise as $y_n\to \infty$.
\end{claim}
\begin{proof} 
By \cref{claim:GreenIsDipole}, for any sequence of vertices $y_n\to \infty$, Cantor diagonalization yields a subsequential pointwise limit of $\g_\o(\cdot, y_n )$, and any such limit is a potential.  Next, for any potential $h$ we have $h(x)\leq \prod_{i=1}^\ell c_{x_{i-1}x_i}^{-1}$, where $x_0,\dots,x_\ell$ is a simple path from $\o$ to $x$, by the argument in the proof of \cref{claim:GreenIsDipole}, so the set of potentials is a closed subset of a compact product set.
Finally, convexity and the last statement are obvious. 
\end{proof}

\begin{remark} Conversely, if   $\g_\o(\cdot, y_n )$ converges pointwise as $y_n\to \infty$, then there is a unique potential; this follows from \cref{cor:PotentialsMixtureDipoles} below.
\end{remark}

Dipoles can be used to show that the uniqueness of the potential does not depend on the choice of  root in the network. 

\begin{claim}\label{claim:uniquenessDoesNotDependOnRoot} Let $(G,c)$ be a recurrent network. Fix two vertices $\o,\tilde{\o}$ and a function $h:V\to [0,\infty)$. Then $h$ is a potential for the network $(G,c,\o)$, if and only if the function
$$ \tilde{h}(\cdot) = h(\cdot) + \g_{\tilde{\o}}(\cdot,\o) - h(\tilde{\o}) \, ,$$
is a potential for $(G,c,\tilde{\o})$.
\end{claim}
\begin{proof} Assume that $h$ is a potential for $(G,c,\o)$. It is clear that $\tilde{h}(\tilde{\o})=0$ and
$$ \Delta \tilde{h}(\cdot) = {\bf 1}_\o(\cdot)+ {\bf 1}_{\tilde{\o}}(\cdot) -{\bf 1}_\o(\cdot)= {\bf 1}_{\tilde{\o}}(\cdot) \, .$$
Therefore $\tilde{h}(X_{\tau_v \wedge t})$ is a martingale bounded below by $-h(\tilde{\o})$. For each $x\in V$, by Fatou's lemma and recurrence
$$ \tilde{h}(x) = \lim_{t\to \infty} \E_x [ \tilde{h}(X_{\tau_{\tilde{\o}} \wedge t})] \geq  \E_x [\liminf _{t\to \infty} \tilde{h}(X_{\tau_{\tilde{\o}} \wedge t}) ] = \tilde{h}(\tilde{\o})=0 \, .$$
The other direction follows similarly.
\end{proof}

\subsection{Basic properties of the $h$-process} \label{sec:basic-h}
Let $h:V\to [0,\infty)$ be a potential. Write $\P^h_\mu$ for the distribution of the $h$-process with initial distribution $\mu$ and $\E^h_\mu$ for the corresponding expectation. Next, we relate the Green kernels of the $h$-process and the network random walk killed at $\o$.  For any finite sequence $\gamma=(v_0,v_1,\ldots,v_k)$   in $V$, we write 
$\P_{v_0}(\gamma):=\prod_{i=1}^k p(v_{i-1},v_i) \,.$
If $\gamma$ is contained in $S_h$, we define 
$\P^h_{v_0}(\gamma)$  similarly using $p^h(v_{i-1},v_i)$. We have the telescoping product
\be \label{eq:hprocessvsSRWSimple} \P^h_{v_0}(\gamma) = \P_{v_0}(\gamma) h(v_k) /h(v_0) \, .\ee
We will use the initial distribution $\mu_h$ defined in \eqref{def:mu}. 
We have
\be \label{eq:hprocessvsSRW} \P^h_{\mu_h} (\gamma) = c_{\o v_0} h(v_0) \P^h_{v_0}(\gamma) =  c_{\o v_0} \P_{v_0}(\gamma) h(v_k) = c_\o \P_\o(\o\gamma) h(v_k)  \, .\ee

\begin{remark}\label{rmk:Sh} By harmonicity of $h$ in $V\setminus \{o\}$, the set $S_h$ is a union of connected components of $V\setminus \{\o\}$ in $G$.    
\end{remark}

\begin{claim} \label{claim:GhMu} Let $h$ be a potential on a recurrent rooted network $(G,c,\o)$. Then the $h$-process   $\{Y_n\}_{n \geq 0}$ is transient. Moreover,  its Green kernel $\G^h(x,y)=\sum_{k \geq 0} \P^h_x( Y_k = y) $, defined for $x,y\in S_h$, satisfies
\be\label{eq:XvsY} \G^h(x,y) = \G_\o(x,y) \frac{h(y)}{h(x)} \, ,\ee
and for all $v\in S_h$
\be\label{eq:Gh} \G^h({\mu_h},v) = h(v) c_v \, ,\ee
where $\G^h({\mu},v):=\sum_x \mu(x) \G^h(x,v)$.

\end{claim}
\begin{proof} Transience will follow    once we show $G^h(x,y)<\infty$. To prove \eqref{eq:XvsY}, we sum \eqref{eq:hprocessvsSRWSimple}  over all paths from $x$ to $y$ that avoid $\o$. Note that by \cref{rmk:Sh}, any such path is contained in $S_h$. Next, \eqref{eq:Gh} follows from \eqref{eq:XvsY} by
$$ \G^h({\mu_h},y) =\sum_v c_{\o v} h(v) \G^h(v,y) = h(y) \sum_v c_{\o v} \G_\o(v,y) = h(y) c_y \, ,$$ using \eqref{eq:dipole11} and the symmetry of $\g_\o(\cdot,\cdot)$ in the last step.
\end{proof}

\begin{claim}\label{lem:hProcessInfty} Let $h$ be a potential for a  recurrent rooted network $(G,c,\o)$ and consider the $h$-process $\{Y_n\}$ started at $v \in S_h$. Then for every $M >0$, 
$$ \P^h_v \big( h(Y_\ell) >M \bigr) \to 1 \quad \text{as} \quad \ell \to \infty \, .$$
\end{claim}
This was also observed in \cite[Lemma 5]{EngelHutch}. 
\begin{proof}
   Recall that $\{X_n\}$ is the network random walk. By \eqref{eq:hprocessvsSRWSimple}, we have 
   $$\P_v^h\big( h(Y_\ell) \le M \bigr)=\E_v\Bigl[\frac{h(X_\ell)}{h(v)} \, {\bf 1}_{\{h(X_\ell) \le M \; \, \text{and} \; \,\tau_\o  \ge \ell\}}  \Bigr] \le \frac{M}{h(v)} \, \P_v( \tau_\o  \ge \ell) \,,$$  
   which tends to $0$ as $\ell \to \infty$ by  recurrence. 
\end{proof}


Every potential $h$ on a recurrent network is unbounded; indeed, if $\{X_n\}$ is the network random walk, then $\{h(X_n)\}$ is a submartingale that does not converge. However, the gradient of a  potential on each edge is bounded by the edge resistance, as the next lemma shows. 


\begin{lemma}\label{lem:lipschitz1} Every potential $h$ on a recurrent rooted network $(G,c,\o)$,  satisfies
$$ |h(x)-h(y)|\leq r_{xy} \, ,$$
for every edge $\{x,y\}\in E$, where $r_{xy}={1 \over c_{xy}}$ are edge resistances.
\end{lemma}
\begin{remark} We will prove a stronger statement in \cref{cor:PotentialLipschitz}, but we find the proof of \cref{lem:lipschitz1} too enjoyable to omit.
\end{remark}
\begin{proof} 
By \cref{claim:GhMu}, for any $v,w\in S_h$ with $\{v,w\}\in E$, we have
$$ \G^h({\mu_h},v) p^h(v,w) = h(v) c_v {h(w) c_{v w} \over h(v) c_v } = h(w) c_{vw} \, .$$
The left hand side of the equality above is the expected number of crossings $N(v,w)$ of the directed edge $(v,w)$ by $\{Y_n\}$. Hence the expected   \emph{net} number of  crossings of the directed edge $(v,w)$ (that is, $N(v,w)-N(w,v)$) equals $c_{vw}(h(w)-h(v))$. 

Next we observe that in a transient network, the expected net  number of  crossings of a directed edge by the network random walk equals the expected net  number of crossings of the same edge by the loop-erased random walk (started at the same initial distribution) because each cycle of the simple random walk has equal probability of being traversed forward or backward.  Since the net number of   crossings of any edge by the loop erased random walk is $-1$, $0$ or $1$, we conclude that $c_{vw}|h(v)-h(w)|\leq 1$. \end{proof}

\subsection{Obtaining potentials from exhaustions}
Another way to generate potentials uses an {\bf exhaustion} of $G$, that is, a sequence of finite sets $\{D_n\}$ such that $o\in D_n$ and $D_n \subset D_{n+1}$, that satisfy $\cup_n D_n = V$. For each $n$, let 
$$\G_n(x,y)=  \sum_{k\geq 0} \P_x(X_k=y, \tau_{D_n^c} > k) \, $$ be the Green kernel of the random walk killed upon exiting the set $D_n$, and write $\g_{n}(x,y)=c_y^{-1} \G_n(x,y)$. The function
\be\label{def:hn} h_n(x) = \g_{n}(\o,\o) - \g_n(x,\o)=\P_x(\tau_{D_n^c} <\tau_\o)\, \g_{n}(\o,\o) \, ,\ee
is harmonic in $D_n \setminus \{\o\}$. It takes the value  $0$ at $\o$, and the value $\g_n(\o,\o)$ at any $x\in V\setminus D_n$. Furthermore, 
$$ \Delta h_n(\o)=\sum_{x \sim \o} c_{\o x} h_n(x) = \sum_{x \sim \o} p(\o,x) [\G_n(\o,\o) - \G_n(x,\o)] = \G_n(\o,\o) - \E_\o \G_n(X_1,\o) = 1\, .$$
Hence, any subsequential point-wise limit of $\{h_n\}$ is a potential. Next we prove a converse for networks with bounded edge resistances. 



\begin{lemma}\label{lem:exhaustion} Let $(G,c,\o)$ be a recurrent     rooted network. Assume that the edge resistances are bounded,
$R:=\sup_{e\in E}  r_e <\infty$, and let $h$ be a potential. Then there exists an exhaustion $\{D_n\}$ of $G$ such that 
$$ \forall  v\in V, \quad h(v) = \lim_{n\to \infty} h_{n}(v) \, ,$$
where $h_n$ are defined in \eqref{def:hn}.
\end{lemma}
\begin{proof} By recurrence, for each $n$ we can choose a ball $B_n$ centered at $\o$ of large enough radius, such that  
$$\P_\o( \tau_{B_n^c} < \tau_\o^+) \leq  \frac{1}{c_\o n^2}   \, .$$
Define
$$ D_n := \{ v \in B_n : h(v) \leq n\} \quad \text{and} \quad  \tau := \tau_\o^+ \wedge \tau_{D_n^c} \,.$$
 Start the network random walk at  $X_0=\o$, and consider the martingale
$M_t:=c_\o h(X_{\tau \wedge t })+{\bf 1}_{\{t=0\}}$ which is bounded by \cref{lem:lipschitz1}. Optional stopping yields
$$1=M_0=\E_\o M_{\tau} \ge  \P_\o\big(h(X_\tau) > n\big)c_\o n   \, ,$$
and by  \cref{lem:lipschitz1},
$$1=M_0=\E_\o M_{\tau} \le \P_\o(X_\tau \in D_n^c) c_\o (n+R)\,. $$

Thus,
$$\frac{1}{c_\o(n+R)} \le \P_\o(X_\tau \in D_n^c) \le \P_\o\big(h(X_\tau) > n\big) + \P_\o( X_\tau \in B_n^c) \le  \frac{1}{c_\o n}+\frac{1}{c_\o n^2}=\frac{n+1}  {c_\o n^2} \,.$$
Since the number of visits to $o$ by the random walk is a geometric random variable,
\be\label{eq:gnooBound} 
n+R \geq g_n(\o,\o)=\big(c_\o \P_\o(X_\tau \in D_n^c)\big)^{-1}  \ge \frac{n^2}{n+1} \ge n-1\,. \ \ee
  
Fix $v \in V$ of graph distance $\ell$ from $\o$, so that $p^\ell(\o,v) >0$. Suppose that $n$ is large enough so  that  $D_n$ contains the shortest path from $\o$ to $v$.  
Observe that
\be \label{eq:vbnd} {1 \over c_\o n^2} \geq \P_\o(X_\tau \in B_n^c) \geq p^\ell(\o,v) \P_v(X_\tau \in B_n^c)\, . 
\ee
Applying optional stopping 
to the bounded martingale $h(X_{t \wedge \tau})$
gives
\be \label{eq:simpopt} h(v)= \E_v(h(X_\tau)) \ge n \P_v(h(X_\tau) > n)   \,. 
\ee
Finally, consider  the bounded martingale
$$\Gamma_t:=h(X_{t \wedge \tau})-h_n(X_{t \wedge \tau}) \,,$$
where we take $X_0=v$. Observe that for $x \in V \setminus D_n$ that has a neighbor in $D_n$,  we have that $h(x) \le n+R$ by \cref{lem:lipschitz1}, and  $h_n(x) =\g_n(\o,\o) \in [n-1,n+R]$ by \eqref{def:hn} and \eqref{eq:gnooBound}. Therefore,
\begin{eqnarray*} |h(v)-h_n(v)|=|\Gamma_0|=|\E_v(\Gamma_\tau)| &\le& (R+1)\,\P_v(h(X_\tau) > n)+(n+R)\, \P_v(X_\tau \in B_n^c) \\[1ex]  &\le& \frac{(R+1) h(v)}{n} +\frac{(n+R) } {p^\ell(\o,v) c_\o n^2} \, ,\end{eqnarray*}
using  \eqref{eq:vbnd} and \eqref{eq:simpopt} in the last step. This concludes the proof.
\end{proof}

\begin{cor} \label{cor:exhaustion} In the setting of \cref{lem:exhaustion}, the network random walks started at $v_1\in S_h$ conditioned to exit $D_n$ before reaching $\o$, converge weakly to the $h$-process as $n\to \infty$.
\end{cor}
\begin{proof} Let $\gamma=(\gamma_1,\dots,\gamma_\ell)$ be a path in $S_h$ with $\gamma_1=v_1$. 
Then
$$ \P_{v_1}(\gamma \mid \tau_{D_n^c} < \tau_{\o}) = \P_{v_1}(\gamma) {h_n(\gamma_\ell) \over h_n(\gamma_1)} \to \P^h(\gamma) \, ,$$
as required.    
\end{proof}

Next we improve \cref{lem:lipschitz1}, replacing edge resistances by \emph{effective resistances}. For background on effective resistance see \cite[Chapter 2]{TheBook}. In recurrent networks, one can define the effective resistance $\Reff(x \lr y)$ between two vertices $x\neq y$ as
$$ \Reff(x \lr y) = \g_x(y,y) = \bigl(c_y P_y(\tau_x <\tau_y^+)\bigr)^{-1} \, .$$

\begin{lemma}\label{lem:1LipschitzReff} Let $(G,c,\o)$ be a rooted network (not necessarily recurrent). Assume that $D\subset V$ satisfies $\o\in D$ and $\P_\o(\exists n \, \, X_n \not \in D)=1$. Then the Green density $\g_{D^c}(\cdot,\o)$ of the random walk killed upon exiting $D$ is $1$-Lipschitz with respect to the effective resistance metric, that is, 
$$ \forall \, x,y\in D \qquad |\g_{D^c}(x,\o) - \g_{D^c}(y,\o)| \leq \Reff(x \lr y) \, .$$
\end{lemma}
\begin{proof} For the random walk $\{X_n\}$ started at $x$, we have
$$ \sum_{t=0}^{\tau_{D^c}} {\bf 1}_{\{X_t = \o\}} \leq \sum_{t=0}^{\tau_{y}} {\bf 1}_{\{X_t = \o\}} + \sum_{t=\tau_y}^{\tau_{y} \vee \tau_{D^c}} {\bf 1}_{\{X_t = \o\}} \, .$$
Taking expectations and using the strong Markov property gives
$$ \G_{D^c}(x,\o) \leq \G_y(x,\o) + \G_{D^c}(y,\o) \, ,$$
which dividing by $c_\o$ yields
$$ \g_{D^c}(x,\o) \leq \g_y(x,\o)+ \g_{D^c}(y,\o) \, .$$
Hence
$$ \g_{D^c}(x,\o) - \g_{D^c}(y,\o) \leq \g_y(x,\o) = \g_y(\o,x) \leq \g_y(x,x) = \Reff(x \lr y) \, ,$$
giving the desired result by interchanging the roles of $x$ and $y$.
\end{proof}

\begin{cor}\label{cor:PotentialLipschitz} Let $(G,c,\o)$ be recurrent rooted network and let $h:V\to [0,\infty)$ be a potential. Then
$$ \forall \, x,y\in D \qquad |h(x)-h(y)| \leq \Reff(x \lr y) \, .$$
\end{cor}
\begin{proof} If the edge resistances are bounded, then this follows directly from Lemmas \ref{lem:exhaustion} and \ref{lem:1LipschitzReff}. Otherwise, we consider the enhanced network $(G^*, c^*) = (V^*, E^*, c^*)$ where $V\subset V^*$ and every edge $e\in E$ with $r_e>1$ is split into $\lceil r_e \rceil$ edges of resistance $r_e/\lceil r_e \rceil$. Observe that for every $x,y\in V$, the effective resistance $\Reff(x \lr y)$ in $(G^*, c^*)$ equals the effective resistance in $(G, c)$. The potential $h$ has a unique extension to a potential $h^*$ on $(G^*, c^*)$ and the result follows from the case of bounded edge resistances.
\end{proof}
\begin{remark}
In the proof above, we could replace \cref{lem:exhaustion}  by     
\cref{cor:PotentialsMixtureDipoles} below.
\end{remark}

\begin{remark} \cref{cor:PotentialLipschitz} can also be obtained from \cite[Proposition 3.6]{BereEngel} together with \cite[Proposition 3.2 (ii)]{BereEngel}.
\end{remark}

\subsection{Preliminaries on transient chains}\label{sec:transient} Since the $h$-process is transient (\cref{claim:GhMu}), we present, following Dynkin \cite{Dynkin69}, a general property of transient chains.   Consider a transient Markov chain on a countable state space $S$ with transition probabilities $\q(x,y)$, and denote by $\G(x,y)$ the Green kernel
$$ \G(x,y) = \sum_{n=0}^\infty \q^n(x,y) \, .$$
Let $\mu$ be a finitely supported probability measure on $S$ representing the starting state for the chain. We write 
$$ \G(\mu,y) = \sum_{x \in S} \mu(x) \G(x,y)$$
and assume that $\mu$ satisfies
\be\label{def:standardMeasure} \G(\mu,y)>0  \hbox{  for all  } y\in S \, .\ee
The {\bf Martin kernel} $M(x,y)$ is now defined as 
\be\label{eq:MartinDef} M(x,y) := \frac{\G(x,y)}{\G(\mu, y)} \, .\ee
The strong Markov property implies that $\G(\mu,y)\geq \P_\mu(\tau_x < \infty) \G(x,y)$, hence
\be \label{eq:Minequality} M(x,y)  \leq \frac{1}{\P_\mu(\tau_x < \infty)} = \frac{\G(x,x)}{\G(\mu,x)} \, .\ee

 Let $D \subset V$ be a finite set and let $L_D$ be the last visit to $D$, that is,
$$ L_D := \sup \{\, n \geq 0 :  Y_n \in D\,  \} \, ,$$
with $L_D:=-\infty$ when this set is empty. Clearly, $\P_x(L_D < \infty)=1$ for any $x\in S$. 

\begin{lemma}\label{lem:ExpectationEq1}  Let $x\in S$ and let $D \subset S$ be a finite set with $x \in D$ and $\mu(D)=1$. Then writing $L=L_D$ we have
$$ \E_\mu [ M(x, Y_{L})] = 1 \, .$$
\end{lemma}
\begin{proof} For any $y\in D$ and $k\geq 0$, we have
$$ \P_x(Y_L = y, L=k) = q^k(x,y) \P_y(L=0) \, .$$
Therefore
$$ \P_x(Y_L=y) = \G(x,y) \P_y(L=0) \, ,$$
whence
\be\label{eq:exit2}
\P_\mu(Y_L=y) = \G(\mu,y) \P_y(L=0) \, .\ee
Thus 
$$ M(x,y) = \frac{\P_x(Y_L=y)}{\P_\mu(Y_L=y)} \, ,$$
provided that $\P_\mu(Y_L=y)>0$. We conclude that 
$$ \E_\mu [M(x,Y_L)] = \sum_{y\in D} \P_\mu(Y_L=y) M(x,y) = \sum_y \P_x(Y_L=y) = 1 \, .$$
\end{proof}

\subsection{Representing potentials as mixture of dipoles}

The next theorem shows that every potential, when restricted to a finite set of vertices $D$ containing the neighbors of $\o$, is a mixture of dipoles $\g_\o(\cdot,v)$ with $v\in D$. 

\begin{thm}\label{cor:PotentialsMixtureDipoles} Let $h$ be a potential for a  recurrent rooted network $(G,c,\o)$. Let $\{Y_n\}_{n \ge 0}$ be the corresponding $h$-process with initial distribution $\mu_h$ defined in \eqref{def:mu}. Let $D\subset V$ be a finite set containing $\o$ such that $\mu_h(D)=1$.  Then for all $x\in D$, we have
\be\label{eq:PotMixture} h(x) = \E_{\mu_h} \g_o(x, Y_L) \, ,\ee
where $L=L_D := \sup \{\, n \geq 0 :  Y_n \in D\, \}$.
\end{thm}

\begin{proof}[Proof of \cref{cor:PotentialsMixtureDipoles}.]
If $x\not\in S_h$, then both sides of \eqref{eq:PotMixture} vanish by \cref{rmk:Sh}. Therefore we may assume that $x\in D \cap S_h$. For all $y \in S_h$ we get from 
  \eqref{eq:XvsY} that $\g^h(x,y)=\g_\o(x,y)h(y)/h(x)$,  and  \cref{claim:GhMu} yields $\g^h(\mu_h,y)=h(y)$.
Thus  \cref{lem:ExpectationEq1}, applied to $D\cap S_h$ (instead of $D$),   implies that
$$1=\E_{\mu_h} \left [ {\g^h(x, Y_{L})\over \g^h(\mu_h,Y_{L})} \right ] = \E_{\mu_h}\left [ {\g_\o(x, Y_{L}) \over h(x)} \right ]\, .$$
\end{proof}

In the next claim, we show that the time reversal of an $h$-process,  run until it exits a finite set, is a network random walk stopped at $o$. This has an application to harmonic measure from infinity (\cref{cor:ExitMeasureHarmonicMeasure}) and to uniform spanning trees, see \cref{lem:Not1Ended}.
\begin{claim}\label{claim:HUntilExit} In the setup of \cref{cor:PotentialsMixtureDipoles}, let $\nu_D$ be the law of $Y_{L}$. Then $(Y_L,\ldots,Y_0,\o)$ has the law of the network random walk with initial distribution $\nu_D$, stopped at $\o$. 
\end{claim}

\begin{proof} We have to show that for any finite path $(v_0, v_1,\ldots,v_\ell)$ in $S_h$, 
\be \label{eq:leftright} \P_{\mu_h}^h\Big ((Y_0,\ldots,Y_L)=(v_0,\ldots,v_\ell)\Big ) = \P_{\nu_D}\Big ((X_0,\ldots,X_{\tau_\o})=(v_\ell,\ldots,v_0,\o)\Big ) \, .
\ee
We write the left hand side using \eqref{eq:hprocessvsSRW} as
$$ c_\o p(v_\o,\o) \Big [ \prod_{i=1}^{\ell} p(v_{i-1},v_i)\Big ]h(v_\ell) \P^h_{v_\ell}(L=0) \, .$$
On the other hand, using \eqref{eq:exit2} from the proof of \cref{lem:ExpectationEq1} and \cref{claim:GhMu}, we deduce that the right hand side of \eqref{eq:leftright} equals
$$ h(v_\ell)c_{v_\ell}\P^h_{v_\ell}(L=0) \prod_{i=0}^{\ell-1} p(v_{\ell-i},v_{\ell-i-1}) p(v_\o,\o) \, .$$
Thus the two sides of \eqref{eq:leftright} are equal by reversibility.    
\end{proof}

\section{Harmonic measures from infinity}\label{sec:Harmonic}

Given a recurrent network $(G,c)$, a finite set $A\subset V$ and a vertex $v$, the {\bf harmonic measure} from $v$ on $A$ is 
$$ \omega_v^A(z) := P_v(X_{\tau_A} =z) \, ,$$
for any $z\in A$. We say that $A$ admits {\bf harmonic measure from infinity} if the limit $\omega_\infty^A(z) = \lim _{v \to \infty} \omega_v^A(z)$ exists for all $z\in A$. 
If this holds for all finite $A \subset V$, then we say that the network $(G,c)$ admits {\bf harmonic measures from infinity}. The following criterion is often useful.

\begin{claim} \label{clm:HarmCriterion}
A recurrent network $(G,c)$ admits harmonic measures from infinity, if and only if there exists a vertex $\o \in V$ and an exhaustion $\{D_n\}$ of $G$ with $o \in D_n$ for all $n$ such that 
 the harmonic measures satisfy
$$q:=\sup_n \max_{v_1,v_2 \in  \partial D_{n+1}} \|\omega_{v_1}^{\partial D_n}-\omega_{v_2}^{\partial D_n}\|_{{\rm TV}}<1 \, .$$

\end{claim}
\begin{proof} 
Given a finite set $A \subset V$,
find $k$ such that $A \subset D_k$ and observe inductively that for each $n\geq 1$ and $v_1, v_2 \notin D_{n+k}$, we have  $$\|\omega_{v_1}^{\partial D_k}-\omega_{v_2}^{\partial D_k}\|_{{\rm TV}} \le q^n \,, \quad \text{whence} \quad
 \|\omega_{v_1}^{A}-\omega_{v_2}^{A}\|_{{\rm TV}} \le q^n \,. $$ 
\end{proof}

\begin{thm} \label{thm:PotentialHarmonicMeasure} A recurrent network $(G,c,\o)$ has a unique potential  if and only if every finite set $A\subset V$ admits harmonic measure from infinity. In this case, the unique potential $h$ and the collection of harmonic measures from infinity  on finite sets can be computed from each other as follows.
\begin{itemize} 
\item[{\rm (a)}] For all $z\in V\setminus\{\o\}$, we have $h(z) = \omega_\infty^{\{\o,z\}}(z) \g_\o(z,z)$. \vspace{2ex}

\item[{\rm (b)}]  If  $\o \in A$ and $z \in A \setminus\{\o\}$, then
$\displaystyle \omega_\infty^A(z) = \frac{\det(M^{h,z})}{\det(M)}$, where $M$ is the invertible matrix 
$M = \big ( \g_\o(x,y) \big )_{x,y \in A \setminus \{\o\}}$,
 and for every $z\in A\setminus\{\o\}$, the matrix
$M^{h,z}$ is  obtained from $M$ by replacing the  $z$-column   with the vector $(h(x))_{x\in A \setminus \{\o\}}$.

\item[{\rm (c)}]  If  $\o \notin A$, choose $\tilde{o} \in A$ and let $\tilde h$ be the unique potential for the rooted network
$(G,c, \tilde o)$, as in \cref{claim:uniquenessDoesNotDependOnRoot}.
To compute $\omega_\infty^A(\cdot)$, use (b) with $\tilde h$ instead of $h$. 
\end{itemize}
\end{thm}

\begin{remark} If $z$ is a vertex of $G$ such that there exists a network automorphism switching between $z$ and $\o$, then $\omega_\infty^{\{\o,z\}}(z)={1/2}$. For any such vertex $z$ we deduce from part (a) of the theorem above that $h(z) = \g_\o(z,z)/2$. In the particular case of $\mathbb{Z}^2$ this is contained in Proposition 4.4.1 of \cite{LawlerLimic}.
\end{remark}
\begin{remark}
Part (b) of the theorem above follows  \cite[Chapters 11 and 14]{Spitzer}.
\end{remark}

To prove this theorem, we first relate  harmonic measure  from a vertex to dipoles.

\begin{lemma}\label{lem:HarmonicMeasureDipoles}Let $(G,c,\o)$ be a recurrent rooted network and let $A \subset V$ be a finite set containing $\o$ with $|A|\geq 2$. Denote by $M$ the square matrix 
$$ M = \big ( \g_\o(x,y) \big )_{x,y \in A \setminus \{\o\}} \, .$$
Then $M$ is invertible. Moreover, for every $v\neq \o$ and $y\in A$,
$$ \omega_v^A(y) = \frac{\det(M^{v,y})}{\det(M)}\, ,$$
where $M^{v,y}$ is the matrix 
obtained from $M$ by replacing the $y$-th column with the vector $(\g_\o(x,v))_{x\in A \setminus \{\o\}}$.
\end{lemma}
\begin{proof} First entrance decomposition gives that for every $y\in A\setminus \{o\}$, 
$$ \g_\o(v,y) = \sum_{x \in A \setminus \{\o\}} 
\omega_v^A(x) \g_\o(x,y) \, .$$
We first show that $M$ is invertible. Indeed, suppose that the coefficients $\lambda_x$ satisfy $\sum_{x \in A \setminus \{o\}} \lambda_x \g_\o(x,y) = 0$ for all $y\in A \setminus \{\o\}$. Define $\psi:V \to \RR$ by
$$ \psi(y) := \sum_{x \in A \setminus \{o\}} \lambda_x \g_\o(x,y) \, .$$
Then $\psi$ is bounded, harmonic off $A$ and vanishes on $A$. Hence by recurrence $\psi \equiv 0$. On the other hand, for all $y\in A \setminus \{\o\}$ we have
$$ \Delta \psi (y) = -\sum_{x \in A \setminus \{o\}} \lambda_x {\bf 1}_x(y) = -\lambda_y \, ,$$
hence $\lambda_y=0$. The result now follows from Cramer's rule.
\end{proof}

\begin{remark} In particular, when $A=\{\o,y\}$, we get
\begin{equation} \label{eq:2PtsA}
    \omega_v^A(y) = \frac{\g_\o(v,y)}{\g_\o(y,y)} \, ,
\end{equation} 
which is also easy to verify directly. 
\end{remark}

\begin{proof}[Proof of \cref{thm:PotentialHarmonicMeasure}] Assume first that there exists a unique potential $h$. Then $h(x) = \lim_{v\to \infty} \g_\o(x,v)$, hence (b) follows from \cref{lem:HarmonicMeasureDipoles} and in particular it shows that every finite set admits harmonic measure from infinity. Furthermore, (a) follows from \eqref{eq:2PtsA}. 

Conversely, assume that every finite set admits harmonic measure from infinity. By \eqref{eq:2PtsA} we deduce that $\g_\o(v,y)=\g_\o(y,v)$ converges as $v\to \infty$ for any $y \in V$. By taking an exhaustion by finite sets $\{D_n\}_{n \geq 1}$, \cref{cor:PotentialsMixtureDipoles} implies there is a unique potential given in (a). \end{proof}

Combining the preceding theorem with the path reversal claim (\cref{claim:HUntilExit}) yields the following folklore corollary.

\begin{cor}\label{cor:ExitMeasureHarmonicMeasure} Let $(G,c,\o)$ be a recurrent network with a unique potential $h$ and $A\subset V$ a finite set so that $\mu_h(A)=1$. Then the last exit distribution $\nu_A$ of the $h$-process on $A$ equals the harmonic measure from infinity on $A$. 
\end{cor} 
\begin{proof} Let $\{Y_n\}$ be the $h$-process started, as usual, from the initial distribution $\mu_h$ and let $D$ be any set containing $A$. By \cref{claim:HUntilExit}, the reversed path $(Y_{L_D},\ldots,Y_0,\o)$ is distributed as network random walk starting from $\nu_D$ and stopped at $\o$. The first vertex in $A$ on this reversed path is $Y_{L_A}$ which has law $\nu_A$.  
Taking $D$ to be a ball of radius $R$ tending to infinity, it follows that harmonic measure from infinity on $A$ also equals $\nu_A$.
\end{proof}


\begin{remark}\label{rmrk:cool2} Let $\{Y_n\}$ be the $h$-process in a recurrent rooted network with a unique potential $h$. Motivated by the preceding corollary, we define the {\bf random walk from infinity to $\o$} in this network to be the process
$\{\widetilde{Y}_{k}\}_{k=-\infty}^0$, where $\widetilde{Y}_{k}:=Y_{-k}$ for $k < 0$. \cref{cor:ExitMeasureHarmonicMeasure} implies that for every finite set $A \subset V$ that contains all neighbors of $\o$,
  harmonic measure from infinity on $A$ is the law of $\widetilde{Y}_{\tau}$, where $\tau :=\min\big\{k<0: \widetilde{Y}_{k} \in A\big\}$.
\end{remark}

In a rooted network that admits a unique potential $h$, we can give a short proof that $h$ tends to infinity, using \cref{cor:ExitMeasureHarmonicMeasure} and \cref{thm:PotentialHarmonicMeasure}. Berestycki and van Engelenburg \cite[Proposition 5.1]{BereEngel} proved this in the special case of unimodular random graphs and asked whether \cref{thm:unique} holds. 

\begin{cor}\label{thm:unique} Let $(G,c,\o)$ be an infinite, connected, locally finite recurrent rooted network that has a unique potential $h$. Then $h$ tends to infinity, i.e., for every $b>0$, the set $\{v \in V : h(v) \leq b\}$ is finite.
\end{cor}

\begin{remark}
    This result is generalized in the next section. 
\end{remark}

\begin{proof}[Proof of \cref{thm:unique}.]
 We will assume that for some $b>0$, there is a sequence $v_n\to \infty$   such that $h(v_n) \leq b$ and obtain a contradiction. By \cref{lem:hProcessInfty}, if $A$ is a large enough ball around $\o$, then the last exit distribution $\nu_A$ of the $h$-process from $A$ satisfies 
\be\label{eq:nuA} \nu_A\Big ( \{ y \in A : h(y) \geq 4b \} \Big ) \geq 1/2 \, .\ee
Optional stopping for the positive martingale $h(X_{t \wedge \tau_A})$, where   $X_0=v_n$, gives 
$$b \ge h(v_n) \ge \E[h(X_{ \tau_A})] \ge 4b\P(h(X_{ \tau_A}\ge 4b))\,,$$
so
\be\label{eq:omega} \omega_{v_n}^A(\{ y \in A : h(y) \geq 4b \}) \leq 1/4 \, .\ee

By \cref{thm:PotentialHarmonicMeasure} the set $A$ admits harmonic measure from infinity, and by \cref{cor:ExitMeasureHarmonicMeasure} this measure equals $\nu_A$, so the measures
$\omega_{v_n}^A$ tend  to $\nu_A$ as $n \to \infty$. Comparing \eqref{eq:nuA} to \eqref{eq:omega} yields the desired contradiction. 
\end{proof}

\section{Existence of potentials tending to infinity} 
\label{sec:MainProof}

In 1936, G.C.~Evans \cite{Evans36} showed that every compact polar set $K$ in $\R^d$ supports a probability measure $\mu$ such that its Newtonian potential $h_\mu$ is infinite on $K$, see also \cite{Tsuji59}. Thus $h_\mu$ is a  harmonic function on $\R^d \setminus K$ that tends to infinity on every sequence converging to $K$. Using Evans' approach, Nakai \cite{Nakai62} proved that every parabolic noncompact Riemann surface $S$ admits a function $h: S \to [0, \infty)$ which is harmonic off a compact set and tends to infinity (i.e., its level sets $\{x \in S: h(x) \le M\}$ are compact). In Nakai's proof, the set $K$ is replaced by the Stone-\v Cech boundary of $S$. The discrete analog of these results is
\cref{conj:main}, which can also be found in our recent note~\cite{NP2025}.

We say that a function $h:V\to [0,\infty)$ {\bf tends to infinity} 
if for all $M>0$, the set $\{ v \in V : h(v) \leq M \}$ is finite. Observe that on $\Z$, the potentials $h_1(n)=\max(n,0)$ for $n\in \Z$ and $h_2(n)=\max(-n, 0)$ do \emph{not} tend to infinity, but any non-trivial convex combination of them does. On $\Z^2$, on the other hand, it is  well known, see \cite[Sections 12 and 31]{Spitzer}, that there is a unique potential $h$ with $h(v) = \Theta(\log |v|)$ as $v \to \infty$. 

\begin{thm}\label{conj:main} On every recurrent rooted network $(G,c,\o)$, there exists a potential tending to infinity.
\end{thm}

\begin{remark}
The argument used by Nakai \cite{Nakai62} in the setting of Riemann surfaces can be adapted to prove \cref{conj:main}, but this argument is non-constructive.  Here we present a more constructive probabilistic approach, using the Von Neumann minimax theorem for finite zero sum games. In \cite{NP2025}, the Sion minimax theorem was used instead. 
\end{remark}


For any integer $R>0$, we set
$ \partial B(\o,R) = \{ w : d(\o,w)=R\} \, ,$
where $d(\cdot,\cdot)$ is the graph distance in $G$.

\begin{lemma}\label{lem:reduction1} Let $(G,c,\o)$ be a recurrent rooted network. Assume that for every $M>0$ there exists an integer $R>0$ and a potential $\psi=\psi_{M,R}$ such that $\psi(v) \geq M$ for any $v\in \partial B(\o,R)$.
Then there exists a potential $h$ that tends to infinity.
\end{lemma}
\begin{proof} First observe that for every $R$ and $M$, if $v \not \in B(\o,R)$, then $\psi(v)=\psi_{M,R}(v) \geq M$. Indeed, let $\{X_t\}$ be the network random walk started at $v$ and let $\tau_R$ be the hitting time of $\partial B(\o,R)$. Then $\psi(X_{t\wedge \tau_R})$ is a positive martingale, so by recurrence and optional stopping,
$$ \psi(v) \geq \E [ \psi(X_{\tau_R})] \geq M \, .$$
For $n\geq 1$, let $R_n$ and $\psi_n$ denote the radii and potentials corresponding to $M=n 2^n$ in the hypothesis of the lemma. By \cref{claim:compact},
$$ h := \sum_{n=1}^\infty 2^{-n} \psi_n$$
is a potential. Next, if $v\notin B(\o,R_n)$, then $h(v) \geq 2^{-n} \psi_n(v) \geq n$. Thus, $\{v : h(v) < n\}$ is finite for every $n$. 
\end{proof}

\begin{remark}\label{thm:conjectureOnTrees} In the particular case where the graph $G$ is a tree $T$, a potential tending to infinity can be constructed explicitly using the preceding lemma. Indeed, let $T_\infty\subset T$ be the union of all simple infinite paths emanating from the root. Let $M> 0$ be arbitrary. Given an integer $R>0$, let $\g_R(\cdot,\cdot)$ be the green density for the network random walk killed on the complement of the ball $B_R:=\{x \in T_\infty : d(x,\o)<R\}$. By recurrence of $T_\infty$, we can choose $R$ such that $\g_R(\o,\o)\geq M$. Set 
$$ \psi(x) = \g_R(\o,\o) - \g_R(x,\o) \, ,$$
for each $x\in B_R \cup \partial B_R$. We now extend $\psi$ to $T_\infty$ as follows. For each $v\in \partial B_R$ choose an infinite simple path $v=v_0, v_1,v_2,\ldots$ such that $v_1$ is contained in $T_\infty \setminus B$. Let $v_{-1}$ be the parent of $v$ in $T_\infty$. We  set $\psi(v_n)= \psi(v)+n(\psi(v)-\psi(v_{-1}))$ for all $n \ge 1$. We have defined $\psi$ on a subtree $T_\infty ' \subset T_\infty$ and we extend by projection: for any vertex $w\in T_\infty$ set $\psi(w)=\psi(w_*)$ where $w_*$ is the closest vertex to $w$ in $T_\infty'$. Thus $\psi$ is a potential on $T_\infty$ such that its values on $T_\infty \setminus B_{T_\infty}(\o,R)$ are larger than $M$. By \cref{lem:reduction1}, there exists a potential $h$ on $T_\infty$ which tends to infinity. We extend $h$ to $T$ by projection.
\end{remark}

For integer $r\geq 1$, let $\Upsilon_r$ be the simplex of probability measures on $\partial B_r$ and define
\begin{eqnarray*} V_1(\infty, r) &:=& \max _{\psi \in \pot} \min _{\eta \in \Upsilon_r} \psi \cdot \eta = \max _{\psi \in \pot} \min _{v \in \partial B_r} \psi(v) \, ,\\ 
V_2(\infty,r) &:=& \min _{\eta \in \Upsilon_r} \max_{\psi\in \pot} \psi \cdot \eta \, .\end{eqnarray*}
These can be interpreted as the values for players $1$ and $2$ of the zero-sum game where player $1$ selects $\psi\in \pot$, player $2$ selects $v\in \partial B_r$ and the payoff is $\psi(v)$. 

\begin{lemma} \label{lem:minimax} For each integer $r\geq 1$ 
$$ V_1(\infty,r) = V_2(\infty,r) \, .$$
\end{lemma}
As observed in \cite{NP2025}, this follows directly from Sion's minimax Theorem, but here we will derive it using approximation by finite games where the values can be calculated via linear programming.

\begin{proof}[Proof of \cref{lem:minimax}] The inequality 
$V_1(\infty,r) \leq V_2(\infty,r)$ clearly holds. Fix $r\geq 1$. For each integer $R>r$, consider the game where player $1$ chooses $w\in \partial B_R$, player $2$ chooses $v \in \partial B_r$ and the payoff is $\g_\o(w,v)$. By von Neumann's minimax Theorem, the values 
\begin{eqnarray*} V_1(R, r) &:=& \max _{\zeta \in \Upsilon_R} \min _{v \in \partial B_r} \g_\o(v, \zeta) \quad \text{and}\\ 
V_2(R,r) &:=& \min _{\eta \in \Upsilon_r} \max_{w \in \partial B_R} \g_\o(\eta, w) \, \end{eqnarray*}
are equal. By \cref{cor:PotentialsMixtureDipoles}, for each potential $\psi$ and each $v\in \partial B_r$, we have 
$$ \psi(v) = \E _{\mu_\psi} \g_\o(v,Y_{L_R})  \, ,$$
where $L_R$ is the last exit time of the $\psi$-process $\{Y_n\}$ from $B_R \cup \partial B_R$. Thus, for any $\eta \in \Upsilon_r$ and $\psi \in \pot$, we have $\psi \cdot \eta \leq \max_{w \in \partial B_R} \g_\o(\eta,w)$. Hence by definition 
$$ V_2(\infty,r) \leq \max_{\psi} \psi \cdot \eta \leq \max_{w \in \partial B_R} \g_\o(\eta,w) \, ,$$
for any $\eta \in \Upsilon_r$. Minimizing over $\eta$ gives $V_2(\infty,r) \leq V_2(R,r)$. On the other hand, for each $R>r$ there exists $\zeta_R \in \Upsilon_R$ such that
$V_1(R,r) = \min_{v \in \partial B_r} \g_\o(v, \zeta_R)$. Pick a sequence of integer $R_n\to \infty$ such that  $\g_\o(\cdot,\zeta_{R_n})$ converge to a potential $h$, so $\lim_n V_1(R_n,r) = \min_{v\in \partial B_r} h(v)$. Hence
$$ V_2(\infty,r) \leq \lim_n V_2(R_n,r) = \lim_n V_1(R_n,r) = \min_{v\in \partial B_r} h(v) \leq V_1(\infty, r) \, .$$

\end{proof}

\subsection{Two lemmas from \cite{NP2025} on the $h$-process}
The first lemma shows that if a sequence of 
dipoles $g_\o(\cdot,v_n)$ converges pointwise to a potential $h$ as $v_n \to \infty$, then the network random walk from $\o$, conditioned to reach $v_n$ before returning to $\o$, converges weakly to the $h$-process. We prove this more generally, for mixtures of dipoles.
\begin{lemma} \label{lem:PathsConverge} Let $(G,c,\o)$ be a recurrent rooted network. Let $R_n\to \infty$ be sequence of integers and for each $n$ let $\eta_{R_n}$ be a probability measure supported on $\partial B(\o,R_n)$. Assume that the sequence $\g_\o(\cdot,\eta_{R_n})$ converges pointwise to a potential $h$ as $n\to \infty$. Then for any finite path $\gamma$ in $S_h$ starting at a neighbor of $\o$, we have
\be\label{eq:pathsConverge} \sum_{v} \eta_{R_n}(v) \P_\o( \o \gamma \mid \tau_{v} < \tau_\o^+) \longrightarrow \P_{\mu_h}^h(\gamma) \quad \text{as}    \quad n \to \infty \,.\ee

\end{lemma}
\begin{proof}
First note that for $v,w$ different from $\o$, we have
$\g_\o(w,v)=\P_w(\tau_v <\tau_\o)\g_\o(v,v)$ and  by path reversal,
$$ c_\o\P_\o(\tau_v <\tau_\o^+)= 
c_v \P_v(\tau_\o <\tau_v^+)= c_v \G_\o(v,v)^{-1} =\g_\o(v,v)^{-1}\,.$$
 Therefore, if $\gamma=(\gamma_1,\dots,\gamma_\ell)$   does not contain the vertex $v$, then  
 \be \label{dipoledoob}
\P_\o( \o \gamma \mid \tau_{v} < \tau_\o^+)=\P_\o( \o \gamma) \frac{ \P_{\gamma_\ell}(   \tau_{v} < \tau_\o) }{\P_{\o}(   \tau_{v} < \tau_\o^+)}=
c_\o \P_\o( \o \gamma)  \g_\o(\gamma_\ell,v) \,.
 \ee
Integrating this with respect $\eta_{R_n}$ gives
$$ \sum_v \eta_{R_n}(v)\P_\o( \o \gamma \mid \tau_{v} < \tau_\o^+) = c_\o \P_\o( \o \gamma)  \g_\o(\gamma_\ell,\eta_{R_n}) \to 
c_\o \P_\o( \o \gamma)  h(\gamma_\ell) \, ,$$
as $n \to \infty$ by our hypothesis. This yields the desired result by  \eqref{eq:hprocessvsSRW}. 
\end{proof}

Our second lemma is a simple combination of optional stopping and path reversal.

\begin{lemma}\label{lem:dipolePotential}
Let $h$ be a potential for a  recurrent rooted network $(G,c,\o)$ and let $v\neq \o$ be a vertex. Then the network random walk $\{X_t\}$ satisfies 
$$ \P_\o( h(X_j) \geq M \mid \tau_v < \tau_\o^+) \leq {h(v) \over M} \, ,$$
for all $M>0$ and integer $0<j< d(\o,v)$.
\end{lemma}
\begin{proof} We may assume that $h(v)< M$. We have
\be\label{eq:beforeReversing}  \P_\o(h(X_j) \geq M \mid \tau_{v} < \tau_\o^+) \leq  \P_\o(T_{M} < \tau_{v} \mid \tau_{v} < \tau_\o^+) \, ,\ee
where $T_M = \inf \{t > 0: h(X_t) \geq M\}$. Recall that $c_\o\P_\o(\tau_{v} < \tau_\o^+) = {c_{v}}\P_{v}(\tau_\o < \tau_{v}^+)$ by path reversal. Similarly, 
$$c_\o \P_\o(T_{M} < \tau_{v} < \tau_\o^+) = c_v \P_{v}(T_{M} < \tau_{\o} < \tau_{v}^+) \, ,$$
so
$$\P_\o(T_{M} < \tau_{v} \mid \tau_{v} < \tau_\o^+) = \P_{v}(T_{M} < \tau_{\o} \mid \tau_{\o} < \tau_{v}^+) \, .$$
Consider the random walk started at $v$ and write $L = \max\{t < \tau_\o : X_t = v\}$. Then $X_{L},\ldots, X_{\tau_\o}$ is distributed as a random walk started at $v$, conditioned on $\tau_\o < \tau_{v}^+$, and stopped at $\o$. Hence
\be\label{eq:DPstep2} \P_{v}(T_M < \tau_{\o} \mid \tau_{\o} < \tau_{v}^+) = \P_{v}(\exists s\in (L,\tau_\o) : h(X_s)\geq M ) \leq \P_{v}(T_M < \tau_o) \, .\ee
To bound the last probability we note that $S_t=h(X_{t \wedge \tau_o})$ is a non-negative martingale, so by optional stopping
$$ h(v) =  \E_v S_0 \geq \E_v S_{T_M \wedge \tau_\o} \geq M \P_{v}(T_M < \tau_\o) \, .$$ 
Combining this with \eqref{eq:beforeReversing} and \eqref{eq:DPstep2} concludes the proof. \end{proof}

\subsection{Proof of \cref{conj:main}}
We apply \cref{lem:reduction1} and \cref{lem:minimax}.
 Given $M>0$, we must show that there is an integer $r>0$,
 such that 
\be \label{minigoal}
 V_2(\infty,r)=V_1(\infty,r) \ge M \, .
\ee
 For each integer $r>0$, we can choose $\eta_r \in \Upsilon_r$  such that every potential $\psi$ satisfies 
\be\label{eq:minimax} \psi \cdot \eta_r \le V_2(\infty,r) \, .\ee
By \cref{claim:GreenIsDipole} there is a sequence of integers $r_n\to \infty$ such that $\g_\o(\cdot,\eta_{r_n})$ converges pointwise to a potential $h$. By \cref{lem:hProcessInfty}, there exists an integer $\ell$ such that 
$$ \P_{\mu_h}^h(h(X_\ell) \geq 2M) \geq 2/3 \, .$$
Next we apply \cref{lem:PathsConverge} and sum \eqref{eq:pathsConverge} over all paths $\gamma=(\gamma_1,\ldots,\gamma_\ell)\subset S_h$ such that $h(\gamma_\ell)\geq 2M$. We obtain
$$ \lim_{n\to \infty} \sum_v \eta_{r_n}(v) \P_\o(h(X_\ell) \geq 2M \mid \tau_v 
< \tau_\o^+)\geq  2/3 \, .$$ \cref{lem:dipolePotential} implies that 
$$ \P_\o(h(X_\ell) \geq 2M \mid \tau_v 
< \tau_\o^+) \leq {h(v) \over 4M} \, ,$$
so there exists $n$ such that
$$   \eta_{r_n} \cdot \frac{h}{2M}  \geq   {1\over 2} \, .$$
Using \eqref{eq:minimax}  we obtain \eqref{minigoal} with $r=r_n$,  as needed.
\qed

\section{Convergence of dipoles along the $h$-process} \label{sec:convergence}

Given a sequence of vertices $y_n \to \infty$ in a recurrent rooted network $(G,c, \o)$, every subsequential pointwise limit of the dipoles $\g_\o(\cdot, y_n)$ is a potential. However, not every potential is a limit of dipoles; consider the potential $x \mapsto |x|/2$ on $(\mathbb Z, {\bf 1}, \o)$. It is well known that every potential is a limit of convex combinations of dipoles; see, e.g., \cref{cor:PotentialsMixtureDipoles}. 
Our next theorem provides a significant strengthening of this fact.

\begin{thm}\label{thm:convergence} Let $h$ be a potential on a recurrent rooted network $(G,c, \o)$, and let $\{\Y_n\}_{n \geq 0}$ be the $h$-process above with initial distribution 
\be\label{def:mu}\mu_h(v) = c_{\o v} h(v) \quad \text{for all} \; \; v \sim \o  \, .\ee Then almost surely, the dipoles $\g_\o(\cdot, \Y_n)$ converge pointwise to a random potential $H$,
$$ \g_\o(\cdot, \Y_n) \longrightarrow H(\cdot) \, ,$$
and $\E H(x) = h(x)$ for every $x\in V$. If $h$ is an extremal\footnote{$h$ is extremal if it cannot be represented as a non-trivial convex combination of potentials.} potential, then  $H=h$ almost surely.
\end{thm}

The following corollary shows the power of the almost sure convergence in the preceding theorem; convergence along a subsequence would not suffice.

\begin{cor}\label{cor:FiniteIntersection} Let $h_1$ and $ h_2$ be two distinct extremal potentials on a recurrent network $(G,c,\o)$ and let $\{Y_n\}_{n \geq 0}$ and $\{Z_n\}_{n \geq 0}$ be the $h_1$-process and $h_2$-process, respectively (no independence is assumed). Then their trajectories intersect finitely often a.s., that is, the set $\{ (m,n) : Y_m = Z_n \}$ is almost surely finite.
\end{cor}

\begin{proof} Let $x\in V$ be such that $h_1(x) \neq h_2(x)$.  By \cref{thm:convergence}, we almost surely have
$$ \g_\o(x, Y_n) \to h_1(x) \quad \text{and} \quad \g_\o(x, Z_n) \to h_2(x) \, .$$
Since  $\{Y_n\}$ and $\{Z_n\}$ are transient, the desired statement follows.
\end{proof}

The key to the proof of \cref{thm:convergence} is the following theorem on transient chains which is due to Doob \cite{Doob59} and Hunt \cite{Hunt60}. For completeness we provide its proof below, following the exposition by Dynkin~\cite{Dynkin69}. Another exposition can be found in Kemeny and Snell \cite{KemenySnell}. 
Recall the definition of the Martin kernel in \eqref{eq:MartinDef} from \cref{sec:transient}.

\begin{thm} \label{thm:dynkin} Let $\Y_n$ be a transient Markov chain governed by transition probabilities $\q$ and initial distribution $\mu$ satisfying \eqref{def:standardMeasure}. Then for every $x\in S$, the sequence $M(x, Y_n)$ converges almost surely as $n \to \infty$. Furthermore, $\E M(x,Y_n) \to 1$.
\end{thm}

\begin{remark} \cref{thm:dynkin} does not assume reversibility. 
\end{remark}

For the proof of \cref{thm:dynkin} it will be convenient to add a state $*$ to our state space and set  $M(x,*):=0$ for any $x\in S$. Furthermore, let $Y_{-k}:=*$ for $k>0$. Lastly, for a finite set $D\subset S$ with $\mu(D)=1$, we write $L=L_D$ for the last exit time from $D$ (recall that $L=-\infty$ if the chain never visits $D$).

\begin{lemma} \label{lem:independent} Let $D \subset S$ be a finite set with $\mu(D)=1$. Then the process $\{Y_{L-n}\}_{n\geq 0}$ has the Markov property with respect to the filtration 
\be \label{filter} 
\F_n = \sigma(Y_L, \ldots, Y_{L-n}) \, ,
\ee
that is,  
$$ \P_{\mu}(Y_{L-(n+1)}= z_{n+1} \mid \F_n) = \P_{\mu}(Y_{L-(n+1)}= z_{n+1} \mid Y_{L-n}) \, ,$$
for all $n\geq 0$ and $z_{n+1} \in S \cup \{*\}$.
\end{lemma}
\begin{proof} For any $z_0,\ldots, z_n\in S$ we have
\begin{eqnarray}\label{eq:long} \P_{\mu}(Y_{L-n}=z_n, \ldots, Y_L=z_0) &=& \sum_{\ell \geq n} \P_{\mu}(L=\ell,  Y_{\ell-n}=z_n, \ldots, Y_\ell=z_0)
\nonumber \\ &=& \sum_{\ell \geq n} q^{\ell-n}(\mu,z_n) q(z_n,z_{n-1})\cdots q(z_1,z_0) \P_{z_0}(L=0)\nonumber \\ &=& \G(\mu,z_n)  q(z_n,z_{n-1})\cdots q(z_1,z_0) \P_{z_0}(L=0) \, .
\end{eqnarray} Note that if $z_0\not\in D$, then both sides of the above vanish. When the right hand side of \eqref{eq:long} is non-zero, we can write the same formula  for $z_0,\ldots, z_{n+1}\in S$ and obtain that
\be\label{eq:cond1} \P_{\mu}( Y_{L-(n+1)}= z_{n+1} \mid Y_{L-n}=z_n, \ldots, Y_L=z_0 ) = \frac{\G(\mu,z_{n+1})q(z_{n+1},z_n)}{\G(\mu,z_n)} \, .\ee
If $z_0,\ldots ,z_n\in S$ and the expression in \eqref{eq:long} is nonzero but $z_{n+1}=*$, then the left hand side of \eqref{eq:cond1} equals $\mu(z_n)/\G(\mu,z_n)$. Lastly, if $z_0,\ldots, z_k\in S$ for $k<n$ but $z_{j}=*$ for $j\in [k+1,n]$, then the left hand side of \eqref{eq:cond1} equals $1$ if $z_{n+1}=*$ and $0$ otherwise.
\end{proof}


\noindent {\bf Proof of \cref{thm:dynkin}.}  Fix $x\in S$ and let $D\subset S$ be a finite set with $\mu(D)=1$. Next, we will verify that the process
$ S_n : = M(x, Y_{L - n})  \, $ is an $\{\F_n\}$-supermartingale, where $\F_n$ was defined in \eqref{filter}.
Indeed, for each $z_n\in S$, on the event $\{Y_{L-n} = z_n\}$, we have by \eqref{eq:cond1} that
\begin{eqnarray*} \E [S_{n+1} | \F_{n}] &=& \sum_{z_{n+1}\in S} \frac{\G(\mu,z_{n+1})q(z_{n+1},z_n)}{\G(\mu,z_n)} M(x,z_{n+1}) = \sum_{z_{n+1}\in S} \frac{q(z_{n+1},z_n)\G(x,z_{n+1})}{\G(\mu,z_n)}  \\ &\leq& \frac{\G(x,z_n)}{\G(\mu,z_n)} = S_n \,.
\end{eqnarray*}
(To justify the inequality above, note that $\G(x,z_{n+1})q(z_{n+1},z_n) $ is the expected number of crossings of the directed edge $(z_{n+1},z_n)$ by the chain started at $x$.)   Also, on the event $\{Y_{L-n}=*\}$, we have $\E [S_{n+1} | \F_{n}]=0$. Thus $\{S_n\}$ is a  supermartingale. Doob's upcrossing lemma \cite[108]{Williams} gives that for any $b>a>0$,  the number  $U_n[a,b]$ of upcrossings of the interval $[a,b]$ by the sequence $\{S_j\}_{j=0}^n$, satisfies
$$ (b-a)\E U_n[a,b] \leq \E[(S_n-a)^-]= \E[\max(a-S_n,0) ]\leq a \,,$$
where  the last inequality holds since $S_n \geq 0$. Hence $\E U_\infty[a,b] \leq {a \over b-a}$. We deduce that the expected number of downcrossings of the interval $[a,b]$ by the sequence $\{M(x,Y_{j})\}_{j=0}^L$ is at most ${a \over b-a}$. Replacing the set $D$ by a sequence of sets exhausting $S$ and using the monotone convergence theorem, gives that the expected number of downcrossings of $[a,b]$ by $\{M(x,Y_{j})\}_{j=0}^\infty$ is at most ${a \over b-a}$. Hence almost surely the number of downcrossings of any rational interval is finite; furthermore, by \eqref{eq:Minequality} the sequence $\{M(x,Y_{j})\}_{j=0}^\infty$ is bounded, concluding the proof that $M(x,Y_n)$ converges almost surely. Denote this limit by $M_\infty(x)$. 

Lastly, to show that $\E M(x,Y_n) \to 1$, it suffices to show that $\E M_\infty(x)=1$, by \eqref{eq:Minequality} and the bounded convergence theorem. We take a sequence of finite subsets $D_k \subset S$ such that $x\in D_k$ and $\mu(D_k)=1$ for all $k$ and $\cup_k D_k = S$. Then by \cref{lem:ExpectationEq1} we have  
$$ \E [ M(x,Y_{L_k}) ] = 1 \, ,$$
where $L_k = L_{D_k}$. We are done since $M(x,Y_{L_k})$ converges to $M_\infty(x)$ almost surely. \qed \\

\noindent {\bf Proof of \cref{thm:convergence}.} We apply the previous theorem with transition probabilities $\q=p^h$ and state space $S_h = \{v : h(v)>0\}$. We write $\G^h$ and $M^h$ for the corresponding Green and Martin kernels. The initial state is distributed as ${\mu_h}(v) = c_{\o v} h(v) $ which satisfies $\G^h({\mu_h},y)>0$ for any $y\in S_h$. Indeed, given $y\in S_h$, the connected component of $y$ in $V\setminus \{\o\}$ is a subset of $S$ (by harmonicity of $h$ on $V\setminus \{\o\}$) and contains a vertex $v$ for which ${\mu_h}(v)>0$. By \cref{claim:GhMu}, the Martin kernel for the $h$-process satisfies
\be\label{eq:MartinKernelIs} M^h(x,y) = {\G^h(x,y) \over \G^h(\mu_h,y)}= \frac{\g_\o(x,y)}{h(x)} \, .\ee
\cref{thm:dynkin} implies that the dipoles $g_\o(x,\Y_n)$ converge almost surely for every $x\in S_h$. Observe that $g_\o(x,y)=0$ when $x\not \in S_h$ and $y\in S_h$ since every path from $x$ to $y$ must go through $\o$. Thus we may define $H:V\to [0,\infty)$ as the almost sure limit 
$$ H(x) = \lim_{n \to \infty} g_\o(x,\Y_n) \, .$$
The pointwise convergence and   \eqref{eq:dipole11} imply that $H$ is harmonic off $\o$ and   $\Delta H(\o)=1$;  Thus, $H$ is a potential almost surely. 

By \cref{thm:dynkin} and \eqref{eq:MartinKernelIs}, we deduce that $\E g_\o(x,Y_n) \to h(x)$. The inequality $g_\o(x,y) = g_\o(y,x) \leq g_\o(x,x)$  and the bounded convergence theorem imply that for all $x\in V$, we have $\E H(x) = h(x)$ . Finally, if the event $A = \{H(x_0) > h(x_0)\}$ has positive probability for some $x_0 \in V$, then $h_A(x):=  \E [ H(x) \mid A]$ is a potential,  since for all $x \in V$, 
$$  \Delta h_A (x)  = \E [\Delta  H(x) \mid A]={\bf 1}_{\o}(x) \,.$$
Therefore, 
the representation 
$$ h(x) = \P(A) h_A(x) + P(A^c) h_{A^c}(x) \, ,$$
shows that $h$ is not extremal, since $h_A(x_0)>h(x_0)$. Thus, if $h$ is extremal, then $H\equiv h$ almost surely.\qed \\

\section{Potentials and uniform spanning trees}\label{sec:PotentialsUST}
In a finite network $(G,c)$, the weighted uniform spanning tree (UST) is a random spanning tree of $G$ drawn with probability proportional to the product of conductances on its edges. In a  recurrent infinite network, the weighted UST is obtained as a weak limit from an exhaustion by finite subnetworks, see~\cite{BLPS} or \cite[Chapter 10]{TheBook}. 

Recall that an infinite graph has a single end if the complement of any finite set of vertices contains a single infinite connected component. Theorem 14.2 in~\cite{BLPS} states  that if the weighted UST in $(G,c)$ has a single end, then harmonic measure from infinity exists on any finite set of vertices, which by  \cref{thm:PotentialHarmonicMeasure}, means that $(G,c,\o)$ has a unique potential for each $\o$. The following lemma, together with \cref{cor:FiniteIntersection}, gives a stronger statement.

\begin{lemma}\label{lem:Not1Ended} Let $(G,c,\o)$ be a recurrent rooted network and suppose that $h_1,h_2$ are potentials (possibly equal). Assume that $\{Y_n\}$ is an $h_1$-process and that $\{Z_n\}$ is an independent $h_2$-process. If $\P(|\{Y_n\} \cap \{Z_n\}|<\infty) >0$, then the UST on $(G,c)$ is almost surely multiply ended.
\end{lemma}
\begin{proof}
The hypothesis implies that there for some integer $r \ge 1$ and real $\eps>0$,  
\be\label{eq:intersectClose} \P\Big (\{Y_n\} \cap \{Z_n\} \subset B(\o,r) \Big ) \geq \eps \, .\ee
Fix $R\geq r$ and take   $D=B(\o,R)$. Denote by   $L_1 = \max \{k : Y_k \in D\}$
the  last exit time for $\{Y_n\}$ and let $\nu_1$ be the law of the last exit vertex $Y_{L_1}$.  Similarly, let $\nu_2$ be the law of the last exit vertex for $\{Z_n\}$ from $D$. To sample the weighted UST, we will use Wilson's algorithm rooted at $\o$. We first run a network random walk started at a random vertex $w_1$  with law  $\nu_1$ and   stopped at $\o$. After erasing loops, we obtain a simple path $\gamma_1$ from $w_1$ to $\o$. Next, we run a network random walk started at a random vertex $w_2$  with law  $\nu_2$ and   stopped when it intersects $\gamma_1$. Denote by $\gamma_2$ the loop ersure of this walk.  By \cref{claim:HUntilExit} and \eqref{eq:intersectClose}, we deduce that with probability at least $\eps$, the  intersection point of $\gamma_1$ and $\gamma_2$ is in $B(\o,r)$. 

Thus, with probability at least $\eps$, the UST has a vertex in $B(\o,r)$ from which  two disjoint simple paths of length at least $R-r$ emanate. Since $R\geq r$ was arbitrary, we conclude that with probability at least $\eps$,  the UST has multiple ends; by Theorem 8.3 in  \cite{BLPS}, 
this probability is 1. \end{proof}

\begin{remark}\label{rem:manyends} The preceding lemma can be strengthened as follows. Suppose $(G,c,\o)$ is a recurrent rooted network and $\{h_j\}_{j=1}^k$ are potentials (possibly equal). 
Sample    independent $h_j$-processes for  $j=1, \ldots,k$.   If the event that all pairwise intersections of their trajectories are finite has positive probability, then the UST on $(G,c)$ has at least $k$ ends almost surely. 
This extension is proved by repeating the argument used above:  We start a network random walk from  a vertex $w_3$  with law  $\nu_3$ (the exit measure from $D$ by the $h_3$ process) and   stopped when it intersects $\gamma_1 \cup \gamma_2$, etc.
\end{remark}

 \bigskip 
 
 Let $(G,c,\o)$ be a recurrent rooted network with a unique potential $h$. In \cite[Propositions 7.11 and 7.12]{BereEngel} it is shown that under additional assumptions (either unimodularity or planarity) two independent trajectories of the $h$-process must have infinite intersection almost surely. 

Next we show that these assumptions cannot be removed. In particular, we construct a recurrent rooted network where the potential $h$ is unique, but a.s., two independent $h$-process trajectories have finite intersection. 
First, let $d \ge 5$ and consider $\mathbb{Z}^d$ with the usual graph structure and unit edge conductances. For any vertex $v \in \mathbb{Z}^d$, let $\varphi(v) := \P_v(\tau_\o < \infty)$. 

\begin{prop}\label{prop:counterexample1} Fix $d\geq 5$ and for each edge $\{x,y\}$ in  $\mathbb{Z}^d$, let $c_{xy}:=\varphi(x) \varphi(y)$.  The network $(\mathbb{Z}^d, c,\o)$ is recurrent, has a unique potential $h$, and two independent trajectories of the $h$-process have finite intersection almost surely. Therefore, the weighted UST on this network is multiply ended almost surely.  
\end{prop}

\begin{proof} 
Denote by $\mathbb{Q}_v$ the law of the network random walk on $(\mathbb{Z}^d, c)$, started at $v$, and by $q(x,y)$ the corresponding transition probabilities.  
For every two adjacent vertices $x,y$ such that $x \ne \o$, we have (using harmonicity of $\varphi$ at $x$ in the third equality) that
\be \label{qxy}
q(x,y)=\frac{c_{xy}}{\sum_{z}c_{xz}}=\frac{\varphi(y)}{\sum_{z\sim x}\varphi(z)}=\frac{\varphi(y)}{ 2d\varphi(x)}={p(x,y) \varphi(y) \over \varphi(x) } \,,
\ee
where $p(x,y)$ are the transition probabilities for simple random walk on $\mathbb{Z}^d$.

Next, we show that
  that the network $(\mathbb{Z}^d, c)$ is recurrent. Indeed,  for every vertex $v \neq \o$, 
  summing over paths from $v$ that reach $o$ only in their final step, we obtain
 $$ \mathbb{Q}_v ( \tau_\o < \infty ) = {\varphi(\o) \over \varphi(v)} \P_v(\tau_\o < \infty)  = \varphi(\o)= 1 \, .$$

It is straightforward that if $h$ is a potential for the rooted network $(\mathbb{Z}^d, c,\o)$, then $h \varphi$ is a potential for $(\mathbb{Z}^d, 1, \o)$. Let $\g(x,y)$ be the Green density on $(\mathbb{Z}^d, 1)$. For any potential $H$ on $(\mathbb{Z}^d, 1, \o)$, the function $H(x) + \g(\o,x)$ is a positive  harmonic function on $(\mathbb{Z}^d, 1)$,  hence it must be constant. Therefore,   $H(x)= \g(\o,\o) - \g(\o,x)=\g(\o,\o)(1-\varphi(x))$ for all $x \in \mathbb{Z}^d$. We deduce that the network $(\mathbb{Z}^d, c,\o)$ also has a unique potential $h(x)=\g(\o,\o)(1-\varphi(x))/\varphi(x)$.

Let $\mathbb{Q}_\mu^h$ be the corresponding $h$-process with initial distribution $\mu=\mu_h$ defined in \eqref{def:mu}. Note that $\mu_h$ is just the uniform distribution on the $2d$ neighbors of $\o$. Let $\gamma=(v_1,\ldots,v_k)$ be a path in $\Z^d$ that starts at a neighbor of $\o$ and does not visit $\o$, then by \eqref{eq:hprocessvsSRW}
\begin{eqnarray*}\mathbb{Q}_\mu^h(\gamma) &=& c_\o \mathbb{Q}_\o(\o\gamma)h(v_k) = c_\o \mathbb{Q}_\mu(\gamma)h(v_k) = c_\o \P_\mu(\gamma) {\varphi(v_k) \over \varphi(v_1)} h(v_k) \\ &=& c_\o \P_\mu(\gamma) {\g(\o,\o)(1-\varphi(v_k)) \over \varphi(v_1)} 
  = \P_\mu(\gamma) {
  1-\varphi(v_k) \over 1- \varphi(v_1)} \, ,
\end{eqnarray*}
since $c_\o=2d\varphi(v_1)$ and $\g(\o,\o)=\G(\o,\o)/2d={1 \over 2d(1-\varphi(v_1))}$. This proves that the $h$-process on $(\mathbb{Z}^d, c,\o)$ is precisely a simple random walk on $(\mathbb{Z}^d, 1, \o)$ started at a uniformly chosen neighbor of $\o$, conditioned on the positive probability event of never visiting $\o$. 

To conclude, recall that for $d\geq 5$, the traces of two independent simple random walks on $\mathbb{Z}^d$ have finite intersection almost surely, hence, the same conclusion holds for the traces of two independent $h$-processes on $(\mathbb{Z}^d, c,\o)$. \cref{lem:Not1Ended} finishes the proof. \end{proof}
\begin{remark}\label{rem:manyends2}
It follows from \cref{rem:manyends} that in the preceding example, the UST has infinitely many ends.
\end{remark}
 The network presented in \cref{prop:counterexample1}  has conductances tending to zero on the edges. 
For completeness, we also describe an unweighted recurrent rooted graph $(G,\o)$ with a unique potential, where the UST has multiple ends almost surely. 
Start with an infinite binary tree $T_2$ with root $\o$. For each $k \ge 1$, identify level $4k$ of $T_2$ with a 4-dimensional face $F_k\cong \{1,\dots ,2^k\}^4$ of a 7-dimensional discrete cube $Q_k \cong \{1,\dots ,2^k\}^7$. All these cubes are pairwise disjoint. The graph $G_0$  obtained as a union of $T_2$ and all these cubes is transient. Let $\varphi(v)$ be  the probability that simple random walk (SRW) starting from a vertex $v$ of $G_0$ visits the root $\o$. Then for $v$ at level $\ell$ of $T_2$   we have $\varphi(v)=2^{-\ell}$ and $\varphi \equiv 2^{-4k}$ on $Q_k$.
Write $c_{xy}=\varphi(x)\varphi(y)$ for adjacent nodes $x,y$ in $G_0$.
\begin{prop}\label{prop:countergraph}
Define a graph $G$ by subdividing every edge $xy$ in $G_0$ into $c_{xy}^{-1}$ new edges (recall that $c_{xy}^{-1}$ is a positive integer). Then the rooted graph $(G,\o)$ is recurrent with a unique potential, yet the UST in $G$ has multiple ends almost surely. 
\end{prop} 

\begin{proof}
As in \cref{prop:counterexample1}, we see that the rooted network $(G_0,c, \o)$ is recurrent.
  Since the uniform mixing time for SRW in $Q_k$ is $O(2^{2k})$ and from each vertex at level $4k$ of $T_2$ the SRW has positive probability of spending the next $\Omega(2^{3k})$ steps in $Q_k$, the criterion in \cref{clm:HarmCriterion} (with $D_n$ consisting of the first $4n$ levels of $T_2$ and all the attached cubes) implies that $(G_0,c, \o)$ has a unique potential $h_0$. As in the preceding proposition, the $h_0$-process in $(G_0,c, \o)$ is simply SRW in $G_0$, conditioned to avoid $\o$.  
Then $(G,\o)$ has a unique potential $h$, obtained from $h_0$ via linear interpolation, and the $h$-process in $(G,\o)$, watched only when visiting $G_0$, is exactly the $h_0$-process in $(G_0,c,\o)$. 

  Let $\{Y_s^0\}$  be an  $h_0$-process in $(G_0,c,\o)$ that is started at a uniformly chosen neighbor  of $\o$.  Observe that within each cube $Q_k$, this process evolves as an SRW.  It follows from claim \ref{cubeclaim} below that the expected number of intersections in $Q_k$ of $\{Y_s^0\}_{s \ge 0}$ with an independent copy  $\{Y^{\sharp}_t\}_{t \ge o}$  is   $O(2^{-k})$. Moreover, the same bound follows for the expected number of pairs of times $(s,t)$ when $Y_s^0$ and $Y^{\sharp}_t$ occupy adjacent nodes in $G_0$. Therefore, two independent copies of the $h$-process in $G$    will have finite intersection almost surely.
\end{proof}

\begin{claim}\label{cubeclaim}
Let $F$ be a 4-dimensional face $F \cong \{1,\dots ,n\}^4$ of a 7-dimensional discrete cube $Q \cong \{1,\dots ,n\}^7$.
Consider two independent simple random walks $\{X_t\}$ and $\{Z_t\}$ in $Q$, where $X_0,Z_0$ are uniform in $F$.
Let $\tau_X:=\min\{t \ge 1: X_t \in F\}$ and define $\tau_Z$ similarly. Then 
$$ \sum_{t,s \ge 0} \P(X_t=Z_s, \, t \le \tau_X, \, s \le \tau_Z) =O(1/n) \,.
$$
\end{claim}
\begin{proof}
Since the commute time in a 3-dimensional cube of side $n$ is $O(n^3)$, we have  $P_w(\tau_X \ge Cn^3)\le \rho<1$ for some integer $C>0$.
By dividing time into blocks of length $Cn^3$, it suffices to prove that
\begin{equation} \label{eqblock} \sum_{t,s = 0}^{Cn^3} \P(X_t=Z_s) =O(1/n) \,,
\end{equation}
where the first four coordinates of $X_0$ and $Z_0$ are independent and uniform in $F$ and the remaining three coordinates are arbitrary.
Now if $\widetilde{X}_t$ and $\widetilde{Z}_s$ are the last 3 coordinates of $X_t$ and $Z_s$ respectively, then
$\P(\widetilde{X}_t=\widetilde{Z}_s)=O(n^{-3}+(t+s)^{-3/2})$,
whence by considering the number of steps in these 3 coordinates,
$$\P(X_t=Z_s)=O(n^{-4}  (n^{-3}+  (t+s)^{-3/2}))\,.$$
Summing over $t,s \le Cn^3$ yields \eqref{eqblock}. 
\end{proof}

Next we present an example, due to Tom Hutchcroft, showing that the converse to \cref{lem:Not1Ended} is false. Start with a ternary tree $T$ rooted at $\o$ and enumerate the vertices of $T$ by $\{v_1,v_2, \ldots\}$ in a breadth first search, that is, the level (distance to $\o$ in $T$) of $v_n$ is a non-decreasing function of $n$. For each integer $h \geq 1$, replace every edge of $T$ between a vertex of level $\ell$ and its parent at level $\ell-1$ by a path of length $3^{3^\ell}$. Denote the resulting tree by $\widetilde{T}$. We now attach to $\widetilde{T}$ at $\o$ an infinite path $\mathbb{N}=(1,2,3,\ldots)$ and for every $n$, we connect $n$ to $v_n$ by a path of length $3^{3^\ell}$ where $\ell\geq \log_3 n-1$ is the level of $v_n$. Denote the resulting rooted graph with unit conductances by $(G,1,\o)$.

\begin{prop}\label{prop:Tom} The rooted network $(G,1,\o)$ is recurrent and has a unique potential $h$. Furthermore,  two independent $h$-processes have infinite intersection almost surely, and the UST on $(G,1)$ has infinitely many ends almost surely.
\end{prop}
\begin{proof} 
Starting from the vertex $m\in \mathbb{N}$, the probability of visiting the pair $\{m-1,m+3^m+1\}$ before reaching the tree $\widetilde{T}$ is at least $1-C_13^{-m/3}$. Indeed, the effective conductance between $\{m, m+1,\ldots, m+3^m  \}$ and $\{v_m,v_{m+1}, \ldots,v_{m+3^m} \}$ is at most $C_13^{-m/3}$ by the parallel law. 
Therefore, the probability that a simple random walk from $m$ visits $m-1$ before reaching the tree $\widetilde{T}$ is at least
$1-C_2 3^{-m/3}$.
We conclude that, starting from $m$,  the probability of visiting $r<m$ before reaching the tree $\widetilde{T}$ is at least $1-C_3 3^{-r/3}$. 

On the other hand, starting from a vertex $v_m$, the probability of visiting $m$ before reaching either the parent or children of $v_m$ in $T$ is ${1/2}+o_m(1)$, since the paths from $v_m$ to its parent and to $m$ have the same length, while the three paths from $v_m$ to its children are much longer. Hence, for every fixed $r$, the probability that a random walk started at   $v_m$   visits $\lfloor \sqrt{m} \rfloor$ before reaching $B_G(\o,r)$ is $1-o_m(1)$ as $ m \to \infty$.

We deduce from these two estimates  that for any integer $r>0$, there exists $M(r)$  such that every vertex $w$  outside $B_G(\o,M(r))$ satisfies
\be \label{cool}
\P_w( \tau_r=\tau_{B_G(\o,r)}< \infty)=1-o_r(1)\,.
\ee
In particular, this implies that $G$ is recurrent. Furthermore, by \cref{clm:HarmCriterion} harmonic measure from infinity exists 
and \cref{thm:PotentialHarmonicMeasure} implies that the network admits a unique potential $h$. By combining \eqref{cool} and    \cref{cor:ExitMeasureHarmonicMeasure}, we infer that the probability that two independent $h$-processes visit the vertex $r$ is at least $1-o_r(1)$, hence by Borel-Cantelli, their intersection is infinite almost surely.

To study the UST on $G$ recall the stacks representation of Wilson's algorithm \cite[Section 4.1]{TheBook}. We use it to define an independent bond percolation process $H$ on $T$ as follows. For $n=1,2,3,\ldots$ we run a simple random walk in $G$ starting at $v_n$ using the stacks (we discard from the stacks the arrows we used), stopped when reaching $\mathbb{N}$ or $v_n$'s neighbors in $T$. Suppose $w$ is the parent of $v$ in $T$. We declare the edge $vw \in E(T)$ \emph{open} in $H$, if the random walk above, started at $v$, visits $w$ before visiting $\mathbb{N}$ or the children of $v$ in $T$. The probability that an edge $vw \in E(T)$ is open in $H$ is ${1/2}+o(1)$ as $v\to \infty$. Thus, $H$ is a supercritical percolation processes on $T$, whence with positive probability the cluster of $\o$ in $H$ is infinite and has uncountably many ends. We note that if $v\in T$ belongs to the cluster of $\o$ in $H$, then the path from $v$ to $\o$ in $\widetilde{T}$ is contained in the UST of $G$. Thus, the UST of $G$ has uncountably many ends with positive probability. Finally, this probability is $1$ by \cite[Lemma 8.3]{BLPS}.     
\end{proof}

\section*{Acknowledgements} The authors are supported by ERC consolidator grant 101001124 (UniversalMap) as well as ISF grants 1294/19 and 898/23. We are grateful to Tom Hutchcroft for allowing us to present his example, described above \cref{prop:Tom}.

\printbibliography

\end{document}